\documentclass[a4paper,12 pt,reqno,twoside]{amsart}

\usepackage{comment}

\usepackage{amssymb}
\usepackage{latexsym}

\theoremstyle{plain}
\newtheorem{propn}{Proposition}[section]
\newtheorem{thm}[propn]{Theorem}
\newtheorem{lemma}[propn]{Lemma}

\theoremstyle{definition}

\theoremstyle{remark}
\newtheorem*{rem}{Remark}
\newtheorem*{rems}{Remarks}
\newtheorem*{eg}{Example}

\newcommand{\ve}{\varepsilon}

\newcommand{\La}{\Lambda}
\newcommand{\vp}{\varpi}

\newcommand{\Kil}{\mathsf{K}}
\newcommand{\kil}{\mathsf{k}}

\newcommand{\init}{\mathfrak{h}}

\newcommand{\FFock}{\mathcal{F}}
\newcommand{\Fock}{\Gamma}
\newcommand{\Exps}{\mathcal{E}}
\newcommand{\Dy}{\mathbb{D}}

\newcommand{\Step}{\mathbb{S}}

\newcommand{\Uonetwo}{U^{(1,2)}}

\newcommand{\onetwoP}{{}^{(1,2)}\!P}
\newcommand{\oneP}{{}^{(1)}\!P}
\newcommand{\twoP}{{}^{(2)}\!P}
\newcommand{\elP}{{}^{(l)}\!P}
\newcommand{\elF}{{}^{(l)}\!F}
\newcommand{\oneF}{{}^{(1)}\!F}
\newcommand{\twoF}{{}^{(2)}\!F}
\newcommand{\elK}{{}^{(l)}\!K}
\newcommand{\oneK}{{}^{(1)}\!K}
\newcommand{\twoK}{{}^{(2)}\!K}
\newcommand{\peeK}{{}^{(p)}\!K}
\newcommand{\peeL}{{}^{(p)}\!L}
\newcommand{\peeM}{{}^{(p)}\!M}
\newcommand{\peeW}{{}^{(p)}\!W}
\newcommand{\peeI}{{}^{(p)}\!I}
\newcommand{\elL}{{}^{(l)}\!L}
\newcommand{\oneL}{{}^{(1)}\!L}
\newcommand{\twoL}{{}^{(2)}\!L}
\newcommand{\elM}{{}^{(l)}\!M}
\newcommand{\oneM}{{}^{(1)}\!M}
\newcommand{\twoM}{{}^{(2)}\!M}
\newcommand{\elW}{{}^{(l)}\!W}
\newcommand{\oneW}{{}^{(1)}\!W}
\newcommand{\twoW}{{}^{(2)}\!W}
\newcommand{\elI}{{}^{(l)}\!I}
\newcommand{\oneI}{{}^{(1)}\!I}
\newcommand{\twoI}{{}^{(2)}\!I}
\newcommand{\elmu}{\mu_l}
\newcommand{\onemu}{\mu_1}
\newcommand{\twomu}{\mu_2}
\newcommand{\elnu}{\nu_l}
\newcommand{\onenu}{\nu_1}
\newcommand{\twonu}{\nu_2}
\newcommand{\elj}{j_l}
\newcommand{\onej}{j_1}
\newcommand{\twoj}{j_2}
\newcommand{\elk}{k_l}

\newcommand{\czerot}{c_{[0,t[}}
\newcommand{\oneczerot}{c^1_{[0,t[}}
\newcommand{\twoczerot}{c^2_{[0,t[}}
\newcommand{\dzerot}{d_{[0,t[}}
\newcommand{\onedzerot}{d^1_{[0,t[}}
\newcommand{\twodzerot}{d^2_{[0,t[}}
\newcommand{\fzerot}{f_{[0,t[}}
\newcommand{\gzerot}{g_{[0,t[}}

\newcommand{\Real}{\mathbb{R}}
\newcommand{\Rplus}{\Real_+}
\newcommand{\Comp}{\mathbb{C}}
\newcommand{\Nat}{\mathbb{N}}
\newcommand{\Int}{\mathbb{Z}}

\newcommand{\ip}[2]{\langle #1, #2 \rangle}
\newcommand{\bip}[3][\big]{#1\langle #2, #3 #1\rangle}
\newcommand{\Bip}[2]{\Big\langle #1, #2 \Big\rangle}
\newcommand{\norm}[1]{\lVert #1 \rVert}

\newcommand{\evol}[1]{( #1 )_{0\leq s \leq t}}
\newcommand{\sgp}[1]{( #1 )_{t\geq 0}}
\newcommand{\bsgp}[1]{\big( #1 \big)_{t\geq 0}}
\newcommand{\Bsgp}[1]{\Big( #1 \Big)_{t\geq 0}}

\newcommand{\ol}{\overline}
\newcommand{\ot}{\otimes}

\newcommand{\otol}{\overline{\ot}\,}

\newcommand{\op}{\oplus}

\DeclareMathOperator{\Dom}{Dom}

\DeclareMathOperator{\Lin}{Lin}

\DeclareMathOperator{\id}{id}

\DeclareMathOperator{\im}{Im}

\newenvironment{rlist}
{

\begin{enumerate}}
{\end{enumerate}}

\numberwithin{equation}{section}
\begin{document}

\title[Quantum stochastic Lie-Trotter product formula]
{A quantum stochastic \\
Lie-Trotter product formula}
\author{J.\ Martin Lindsay}
\address{Department of Mathematics and Statistics \\
Lancaster University \\
Lancaster LA1 4YF \\ UK}
\email{j.m.lindsay@lancaster.ac.uk}
\author{Kalyan B.\ Sinha}
\address{Jawaharlal Nehru Center for Advanced Scientific Research \\
and Indian Institute of Science \\
Bangalore \\ India} \email{kbs jaya@yahoo.co.in}

\subjclass[2000]{Primary 81S25; Secondary 46L53, 47D06}
\keywords{Trotter product formula, quantum stochastic cocycle,
one-parameter semigroup, random unitary evolution, noncommutative probability}



\begin{abstract}
A Trotter product formula is established for unitary quantum
stochastic processes governed by quantum stochastic differential
equations with constant bounded coefficients.
\end{abstract}

\maketitle

\section*{Introduction}
\label{intro}

The aim of this paper is to establish a quantum probabilistic
counterpart to the well-known Trotter product formula for
one-parameter unitary groups and contraction semigroups (\cite{Tr2})
and its forerunner, the Lie product formula for one-parameter
subgroups of Lie groups (see~\cite{Davies}, [$\text{RS}_{1,2}$].
Some years ago K.R.\ Parthasarathy and the second-named author
obtained a stochastic Trotter product formula for unitary-operator
valued evolutions constituted from independent increments of
independent classical Brownian motions (\cite{PSrandomTrotter}).
This predated the founding of quantum stochastic calculus by Hudson
and Parthasarathy (\cite{HuP}). In this paper Brownian increments
are replaced by the fundamental quantum martingales, namely the
creation, preservation and annihilation processes of quantum
stochastic calculus (\cite{Biane}, \cite{Hudson}, \cite{L},
\cite{Meyer}, \cite{Partha}, \cite{SinhaGoswami}), and we prove a
Lie-Trotter type product formula for unitary quantum stochastic
processes on a Hilbert space which satisfy a quantum stochastic
differential equation with constant bounded coefficients. The case
of quantum stochastic differential equations with unbounded
coefficients, and more general kinds of quantum stochastic cocycle
on operator spaces and $C^*$-algebras, will be addressed in the
forthcoming paper~\cite{more Trotter}.


\section{Unitary quantum stochastic cocycles}
\label{Section: unitary quantum stochastic cocycles}

In this section we fix our notations and recall the essential facts
about quantum stochastic differential equations and unitary quantum
stochastic cocycles that we need here.

Let $\kil$ be a complex Hilbert space, with fixed countable
orthonormal basis, which we refer to as the \emph{noise dimension
space}. Write $\FFock_\kil$ for the symmetric Fock space over the
Hilbert space $\Kil:= L^2(\Rplus; \kil)$ and $\vp(f)$ for the
normalised exponential vector $\exp(-\norm{f}^2/2)\, \ve(f)$, $f\in
\Kil$. When $\Rplus$ is replaced by $[s,t[$, we write $\Kil_{[s,t[}$
and $\FFock_{\kil,[s,t[}$ instead; the continuous tensor
decomposition
\[
\FFock_{\kil} =
\FFock_{\kil,[0,s[} \ot
\FFock_{\kil,[s,t[} \ot
\FFock_{\kil,[t,\infty[},
\]
corresponding to the direct sum decomposition $\Kil = \Kil_{[0,s[}
\op \Kil_{[s,t[} \op \Kil_{[t,\infty[}$, is in constant use below.
For a start, a bounded quantum stochastic process on an
\emph{initial} Hilbert space $\init$ with noise dimension space
$\kil$ is a family of operators $\sgp{X_t}$ on $\init\ot\FFock_\kil$
satisfying the adaptedness condition
\begin{equation}
\label{adapted}
 X_t \in B(\init\ot\FFock_{\kil,[0,t[})\ot I_{\FFock_\kil,[t,\infty[}
= B(\init) \otol B(\FFock_{\kil,[0,t[})\ot
I_{\FFock_\kil,[t,\infty[},
\end{equation}
for all $t\in\Rplus$. For $f\in \Kil$, $f_{[s,t[}$ denotes the
function equal to $f$ on $[s,t[$ and zero elsewhere; $c_{[s,t[}$ is
defined similarly, for $c\in\kil$.
 Let
$\Step_\kil$ and $\Step'_\kil$ denote the subspaces of $\Kil$
consisting of step functions, respectively step functions which have
their discontinuities in the dyadic set $\mathbb{D}:= \{ j2^{-n}:
j,n\in\Int_+ \}$, and let $\Exps_\kil$ and $\Exps'_\kil$ be the
(dense) subspaces $\Lin \{\ve(f): f\in\Step_\kil\}$ and $\Lin
\{\ve(f): f\in\Step'_\kil\}$ of $\FFock_\kil$. For evaluation
purposes, we always take the \emph{right-continuous versions} of
step functions.
 The \emph{order} of a function $f\in\Step'_\kil$
 is the least nonnegative integer $N$
such that $f$ is constant on all intervals of the form
$[j2^{-N},(j+1)2^{-N}[$ for $j\in\Int_+$.

The \emph{time-shift} semigroup $\sgp{\Theta^\kil_t}$
of unital *-monomorphisms of $B(\FFock_\kil)$ is defined by
\[
\Theta^\kil_t(X) = I_{\FFock_{\kil,[0,t[}} \ot
\Gamma(\theta_t^\kil)X\Gamma(\theta_t^\kil)^*, \quad t\in\Rplus,
X\in B(\FFock_\kil),
\]
where $\Fock(\theta_t^\kil): \FFock_\kil \to
\FFock_{\kil,[t,\infty[}$ is the unitary (second quantisation)
operator determined by
\[
\Fock(\theta_t^\kil)\vp(f) = \vp(\theta_t^\kil f) \text{ where }
(\theta^\kil_t f)(s) = f(s-t) \text{ for } s \in [t,\infty[.
\]

Let $\{ \La^\mu_\nu : \mu, \nu \geq 0 \}$ denote the fundamental
quantum semimartingales for the noise dimension space $\kil$, with
respect to its fixed orthonormal basis. Then the quantum stochastic
(QS) integral equation
\begin{equation}
\label{qsde} U_t = I_{\init\ot\FFock_\kil} + \int_0^t U_s F^\mu_\nu
\La^\nu_\mu(ds)
\end{equation}
(where summation over the repeated greek indices is understood), has
a unique strongly continuous solution, consisting of unitary
operators on $\init\ot\FFock_\kil$,
 provided that the matrix of bounded operators $[F^\mu_\nu]$ on
 the initial  space $\init$ satisfies the following
 structural relations (\cite{HuP}).
 It must have the block matrix structure
 \[
 \begin{bmatrix}
 K & [M_k] \\ [L^j] & [W^j_k - \delta^j_k]
 \end{bmatrix}
 \]
of an operator $F\in B(\init \op (\init\ot\kil))$, where
$[L^j]$ is the block column matrix of an \emph{arbitrary} operator $L\in
B(\init;\init\ot\kil)$,
$[W^j_k]$ is the block matrix form of a \emph{unitary} operator
$W\in B(\init\ot\kil)$,
$[M_k]$ is the block row matrix of \emph{the} operator
$M = -L^*W\in B(\init\ot\kil;\init)$, and $K=iH
-\frac{1}{2}L^*L$ for a \emph{selfadjoint} operator $H\in B(\init)$, so
that
\[
M_k = - \sum_{j\geq 1} (L^j)^*W^j_k, \ k\geq 0, \text{ and }
K = iH - \frac{1}{2} \sum_{j\geq 1} (L^j)^* L^j.
\]
These structure relations may equivalently be expressed by the
following two identities, for all
$v = (v^\mu)_{\mu\geq 0}$ in
$\init \op (\init\ot\kil) = \bigoplus_{\mu\geq 0} \init$:
\begin{subequations}
\label{structure}
\begin{align}\label{structure a}
& \sum_{\mu,\nu\geq 0} \bip{v^\mu}{\big( (F^\nu_\mu)^* + F^\mu_\nu +
\sum\nolimits_{j\geq 1} (F^j_\mu)^* F^j_\nu\, \big) v^\nu} \ = 0,
\\ & \label{structure b}
 \sum_{\mu,\nu\geq 0} \bip{v^\mu}{\big( (F^\nu_\mu)^* + F^\mu_\nu +
 \sum\nolimits_{j\geq 1} F_j^\mu (F_j^\nu)^*\, \big) v^\nu} \ = 0;
 \end{align}
\end{subequations}
the first corresponds to isometry and
the second to coisometry.

A contractive quantum stochastic process $\sgp{U_t}$ satisfying
\begin{equation}
\label{qs cocycle} U_{s+t} = U_s \Theta_s(U_t), \ \ U_0 =
I_{\init\ot\FFock}, \quad s,t\geq 0,
\end{equation}
where $\bsgp{\Theta_t := \id_{B(\init)} \otol \Theta_t^\kil}$, is
called a \emph{quantum stochastic contraction cocycle}. If
$\sgp{U_t}$ is a QS contraction cocycle then the operators on
$\init$ defined by
\[
\ip{u}{P_t v} = \ip{u\ot \vp(0)}{U_t\, v\ot\vp(0)} \quad
u,v\in\init,\ t\in\Rplus,
\]
define a contraction semigroup $\sgp{P_t}$  on $\init$ known as the
(\emph{vacuum}) \emph{expectation semigroup} of the cocycle, and the cocycle $\sgp{U_t}$
is called \emph{Markov-regular} if its expectation semigroup is
norm-continuous.

\begin{thm}[\cite{LWjfa}]
\label{1.1*}
Let $\sgp{U_t}$ be a unitary quantum stochastic process on $\init$
with noise dimension space $\kil$. Then the following are
equivalent:
\begin{rlist}
\item
$\sgp{U_t}$ satisfies~\eqref{qsde}, for a matrix of bounded
operators $[F^\mu_\nu]$;
\item
$\sgp{U_t}$ is a Markov-regular quantum stochastic cocycle.
\end{rlist}
\end{thm}

The implication (i) $\Rightarrow$ (ii) follows from the form that
solutions of such QS differential equations take, by virtue of the
time-homogeneity of the quantum noises:
\[
I_{\FFock_{\kil,[0,t[}} \ot \Fock(\theta^\kil_t) \La^\mu_\nu[a,b]
\Fock(\theta^\kil_t)^* = \La^\mu_\nu[a+t,b+t],
\]
and the time-independence of the coefficients of the QS differential
equation.

The converse implication (ii) $\Rightarrow$ (i) may be deduced from
the Quantum Martingale Representation Theorem (\cite{PSmartingale})
applied to the regular quantum martingale
\[
\Bsgp{U_t - \int_0^t U_s K \, ds}
\]
in which the operator $K$ is the generator of the expectation semigroup
of $\sgp{U_t}$ (see~\cite{HuL}). However the more powerful method of
proof in~\cite{LWjfa} goes via the following intermediate
characterisation which is of considerable use itself,
as we shall see below:
\begin{rlist}
\item[(iii)]
there are semigroups $\{\sgp{P^{c,d}}: c,d\in\kil \}$ such that, for
all $f,g\in\Step'_\kil$ and $t\in\Rplus$,
\begin{equation}
\label{semigp rep}
\begin{split}
\ip{u\ot\vp(f_{[0,t[})}{V_t\ v\ot&\vp(g_{[0,t[})} =
\\ &
\ip{u} {P_{t_1-t_0}^{f(t_0),g(t_0)} \cdots \
P_{t_{m+1}-t_m}^{f(t_m),g(t_m)}v},
\end{split}
\end{equation}
where $t_0 = 0$, $t_{m+1} = t$ and
$\{ t_1 < \cdots < t_m \}\subset \Dy$
is the (possibly empty) union of the sets of discontinuity of
$f$ and $g$ in the open interval $]0,t[$.
\end{rlist}

\begin{rems}
The matrix of bounded operators $[F^\mu_\nu]$ necessarily satisfies
the structural relations required for unitarity~\eqref{structure}.

The identity~\eqref{semigp rep} is known as the \emph{semigroup
decomposition} and the collection $\{ \sgp{P^{c,d}_t}: c,d\in\kil\}$
as the \emph{associated semigroups} of the cocycle.
Clearly the associated semigroups are determined by
\begin{equation}
\label{assoc sgp}
\ip{u}{P_t^{c,d}v} = \ip{u\ot \vp(c_{[0,t[})}{U_t\,
v\ot\vp(d_{[0,t[})}, \quad u,v\in\init,
\end{equation}
and $\sgp{P^{0,0}}$ is the expectation semigroup of the cocycle.

In fact, each associated semigroup $\sgp{P^{c,d}}$ is itself
the expectation semigroup of another unitary QS cocycle,
namely the cocycle
\[
\Bsgp{U^{c,d}_t:= (I_\init\ot W^c_t)^*\, U_t (I_\init\ot W^d_t)},
\]
where the \emph{Weyl cocycles} are defined by
\[
W^c_t \vp(f) = e^{-i \im \ip{c_{[0,t[}}{f}}
\vp(f + c_{[0,t[}), \quad f\in\Step_\kil,\ c\in\kil,\
t\in\Rplus.
\]

Markov-regularity for a QS contraction cocycle actually implies that
all of its associated semigroups are norm-continuous. In fact, in
terms of the block matrix form of $[F^\mu_\nu]$, the semigroup
$\sgp{P^{c,d}}$ has bounded generator
\begin{equation}
\label{Gcd}
 G_{c,d}:= K + L^c + M_d + W^c_d - \tfrac{1}{2}(\norm{c}^2
+ \norm{d}^2)I_\init,
\end{equation}
where, in terms of basis expansions of $c$ and $d$, the operators
here are defined as follows:
\[
L^c  = \sum_{j\geq 1} \ol{c^j} L^j , \ \
M_d  = \sum_{k\geq 1} d^k M_k  \text{ and }
W^c_d  = \sum_{j,k\geq 1} \ol{c^j}d^k W^j_k,
\]
the convergence here being in the strong operator topology
(see~\cite{LWmathproc}).
\end{rems}
Given a unitary QS cocycle $\sgp{U_t}$, the family $\evol{U_{s,t}:=
\Theta_s(U_{(t-s)})}$ is a \emph{time-homogeneous adapted unitary
evolution}, that is: for all $a \geq 0$ and $0 \leq r\leq s \leq t$:
\begin{rlist}
\item
$U_{s+a,t+a} = \Theta_a (U_{s,t})$;
\item
 $U_{s,t} \in B(\init)\ot I_{\FFock_{\kil,[0,t[}} \otol B(\FFock_{\kil,[s,t[}) \ot
I_{\FFock_{\kil,[t,\infty[}}$;
\item
$U_{r,t} = U_{r,s}\, U_{s,t}$.
\end{rlist}

Conversely, if $\evol{U_{s,t}}$ is such an evolution then
$\sgp{U_t:= U_{0,t}}$ defines a unitary QS cocycle, and it is easily
seen that the passages between QS cocycle and adapted
time-homogeneous evolution are mutually inverse.

The corresponding QS integral equation satisfied by $\evol{U_{s,t}}$
is
\[
U_{r,t} =
I_{\init\ot\FFock_\kil} + \int_r^t U_{r,s} F^\mu_\nu
\La^\nu_\mu(ds).
\]
Adapted evolutions that are not time-homogeneous arise as solutions
of QS differential equations with time-dependent coefficients
$[F^\mu_\nu]$.

\section{Trotter product of quantum stochastic cocycles}
\label{section: Trotter product}

Let $\sgp{U_t^1}$ and $\sgp{U_t^2}$ be two unitary QS cocycles on
the same initial space $\init$, with noise dimension spaces $\kil_1$
and $\kil_2$ having fixed countable orthonormal bases. Suppose
that they are both Markov-regular, equivalently that they satisfy QS
differential equations
\begin{equation}
\label{qsdel} dU_t^l = U_s^l\ \elF^{\elmu}_{\elnu}\,
\La^{\elnu}_{\elmu}(dt), \quad U_0^l = I_{\init\ot \FFock^{(l)}},
\end{equation}
$l=1,2$, for matrices of bounded operators
$[\oneF^{\onemu}_{\onenu}]$ and $[\twoF^{\twomu}_{\twonu}]$
satisfying the structural relations
which guarantee unitarity of the
processes. Here $\FFock^{(1)}$ and $\FFock^{(2)}$ denote the Fock
spaces $\FFock_{\kil_1}$ and $\FFock_{\kil_2}$ respectively.

Our aim is to obtain a unitary cocycle $\sgp{U_t}$ as a Lie-Trotter type
product of the cocycles $\sgp{U_t^1}$ and $\sgp{U_t^2}$, in the same
spirit as that of~\cite{PSrandomTrotter}. To this end let $\kil$ be
the noise dimension space $\kil_1\op\kil_2$, set
$\FFock=\FFock_\kil$, and, by `concatenating' the orthonormal bases for
$\kil_1$ and $\kil_2$ to form an orthonormal basis of $\kil$, let
$[F^\mu_\nu]$ be the matrix of bounded operators on $\init$ having
block matrix form
\begin{equation}
\label{F1 box F2}
\begin{bmatrix}
K & [M_k] \\
[L^j] & [W^j_k - \delta^j_k I_{\init}]
\end{bmatrix}
=
\begin{bmatrix}
\oneK + \twoK & \oneM & \twoM \\
\oneL & \oneW - \oneI & 0 \\
\twoL & 0 & \twoW - \twoI
\end{bmatrix}.
\end{equation}
Here $\elI := I_{\init\ot\kil_l}$ and
\[
[\elF^{\elmu}_{\elnu}] =
\begin{bmatrix}
\elK & [\elM_{\elk}] \\
[\elL^{\elj}] & [\elW^{\elj}_{\elk} - \delta^{\elj}_{\elk}
I_{\init}]
\end{bmatrix}
\]
is the block matrix decomposition of $\elF$, in which
\[
\elK = iH_l - \frac{1}{2} \sum_{j_l\geq 1} (\elL^{\elj})^*\,
\elL^{\elj} \quad \text{ and } (H_l)^* = H_l,
\]
for $l=1,2$. (We are slightly cheating in terms of indices
since if the noise dimension space $\kil_1$ is infinite dimensional
then we cannot \emph{exactly} count $1,2, \cdots , \dim \kil_1, \dim
\kil_1 + 1, \cdots$. However all \emph{is} justified by a proper
indexing, or alternatively by working coordinate-free as
in~\cite{more Trotter}.) Thus, setting $H = H_1 + H_2$,
\begin{align*}
K &=  iH - \frac{1}{2} \sum_{j\geq 1} (L^j)^* L^j \\ &= iH_1 + iH_2
- \frac{1}{2} \sum_{j_1\geq 1} (\oneL^{\onej})^* (\oneL^{\onej}) -
\frac{1}{2} \sum_{j_2\geq 1} (\twoL^{\twoj})^* (\twoL^{\twoj}),
\text{ and }\\
[M_k] &=
\Big[
-\oneL^*\,\oneW
\quad -\twoL^*\,\twoW
\Big]
= \left[- \sum_{j\geq 1} (L^j)^*W^j_k \right].
\end{align*}
Thus $[F^\mu_\nu]$ satisfies the structure relations~\eqref{structure}
for unitarity of the solution of the QS
differential equation~\eqref{qsde}.

For $c^l,d^l \in \kil_l$, let
$\sgp{\elP_t^{c^l,d^l}}$ denote the corresponding associated semigroup
of the cocycle $\sgp{U^l_t}$ ($l=1,2$). For each $n\in\Nat$ define a
unitary process $\sgp{U^{(1,2)}_n (t)}$ as follows:
\[
U^{(1,2)}_n(t) :=
\big( \Uonetwo_{0,2^{-n}} \Uonetwo_{2^{-n},2\cdot2^{-n}} \cdots
\Uonetwo_{t^n_{-1},t^n_0} \big) \Uonetwo_{t^n_0,t},
\quad t\in\Rplus,
\]
where, with $[\,\cdot\,]$ denoting the integer part,
\begin{equation}
\label{tnk}
t^n_k:= 2^{-n} \big( [2^nt] + k \big)
\quad \text{ for } k \in \Int, k \geq -[2^nt],
\end{equation}
and, letting $\Sigma_{2,1}$ denote the tensor flip
$B(\init\ot\FFock^{(2)}\ot\FFock^{(1)}) \to
B(\init\ot\FFock^{(1)}\ot\FFock^{(2)}) = B(\init\ot\FFock)$,
\begin{equation}
\label{Wst}
\Uonetwo_{s,t} := \Theta_s(\Uonetwo_{t-s}), \quad \quad 0\leq s\leq t,
\end{equation}
where
\[
\Uonetwo_t:=
\big( U^1_t \ot I^{(2)} \big) \Sigma_{2,1}\big( U^2_t \ot I^{(1)} \big),
\quad t\in\Rplus.
\]
Here $I^{(l)}$ is the identity operator on $\FFock^{(l)}$ ($l=1,2$),
and we are using the isometric isomorphism
$\FFock^{(1)}\ot\FFock^{(2)}=\FFock$. Also define a family of
contractions on $\init$ by
\[
\ip{u}{\onetwoP^{c,d}_t v} =
\ip{u\ot\vp(c_{[0,t[})}{\Uonetwo_t\, v\ot\vp(d_{[0,t[})},
\quad u,v\in\init,
\]
for $c,d \in \kil$ and $t\in\Rplus$ (cf.\ \eqref{assoc sgp}).

\begin{rems}
For each $t\in\Rplus$, $n\in\Nat$ and $k$ as in~\eqref{tnk},
\[
t^n_0 \leq t^{n+1}_0 \leq t < t^{n+1}_1 < t^n_1 \text{ and }
|t^n_{k+1} - t^n_k| = 2^{-n}.
\]
In particular the sequence $(t^n_1)$ decreases to $t$
and the sequence $(t^n_0)$ is
nondecreasing and converges to $t$.

In general, neither $\sgp{\Uonetwo_t}$ nor $\sgp{U^{(1,2)}_n(t)}$ are
cocycles themselves. However they are both unitary QS processes and the
two-parameter process $\evol{\Uonetwo_{s,t}}$ enjoys the factorisations
\begin{equation}
\label{W factorisation} \Uonetwo_{s,t} \in B(\init) \ot I_{[0,s[} \otol
B(\FFock_{[s,t[}) \ot I_{[t,\infty[}
\end{equation}
in which $I_{[0,s[}$ and $I_{[t,\infty[}$ denote the identity operators on
$\FFock_{[0,s[}$ and $\FFock_{[t,\infty[}$. By
the same token, $\sgp{\onetwoP^{c,d}}$ is typically not a semigroup.
\end{rems}

\begin{lemma}
\label{lemma: W}
Let $\sgp{U^1_t}$ and $\sgp{U^2_t}$ be unitary QS cocycles on $\init$
with noise dimension spaces $\kil_1$ and $\kil_2$ respectively.
Set $\kil := \kil_1\op\kil_2$ and let $\sgp{\Uonetwo_t}$ be as defined above.
Let $t\in\Rplus$, then for
$c = \binom{c^1}{c^2}$,
$d = \binom{d^1}{d^2} \in\kil= \kil_1\op\kil_2$,
\[
\onetwoP^{(c,d)}_t =
\oneP_t^{c^1,d^1}\, \twoP_t^{c^2,d^2},
\]
and, for $f,g\in\Step'_\kil$ and $n$ greater than the orders of both $f$
and $g$,
\begin{multline*}
\ip{u\ot\vp(\fzerot)}{U^{(1,2)}_n (t)\, v\ot \vp(\gzerot)} = \\
\Bip{u}{\big(\onetwoP_{2^{-n}}^{f(0),g(0)}\
\onetwoP_{2^{-n}}^{f(2^{-n}),g(2^{-n})} \cdots \
\onetwoP_{2^{-n}}^{f(t^n_{-1}),g(t^n_{-1})} \big)
\onetwoP_{(t-t^n_0)}^{f(t^n_0),g(t^n_0)}\, v}.
\end{multline*}
\end{lemma}
\begin{proof}
These both follow from factorisations; the first from
\begin{multline*}
\bip{u\ot\vp(\czerot)}{\Uonetwo_t\, v\ot\vp(\dzerot)} = \\ 
\Bip{u\ot\vp(\oneczerot)}{U^1_t\big((E^* U^2_t F\, v) \ot\vp(\onedzerot)\big)},
\end{multline*}
where $E$ and $F$ are the isometric operators $\init \to
\init\ot\FFock^{(2)}$ defined respectively by $v\mapsto
v\ot\vp(\twoczerot)$ and $v\mapsto v\ot\vp(\twodzerot)$; in turn, the second
from~\eqref{W factorisation} and $ \vp(h)
=\vp(h_{[0,s[})\ot\vp(h_{[s,t[}) \ot \vp(h_{[t,\infty[})$, for $h=f,g$.
\end{proof}
We now come to our quantum stochastic product formula.
For its proof we use the following version of the
classical Lie product formula.
For bounded operators $Z_1$ and $Z_2$ on $\init$,
\begin{equation}
\label{original}
\big( e^{hZ_1} e^{hZ_2}\big)^{[t/h]} \to e^{t(Z_1 + Z_2)} \text{ as
} h \to 0,
\end{equation}
in operator norm, uniformly on bounded time intervals
(see e.g.\ Theorem VIII.29 of~\cite{RS1},
where the proof is obviously valid for operators).
\begin{thm}
\label{thm: qs trotter} Let $\sgp{U^1_t}$ and $\sgp{U^2_t}$ be
unitary QS cocycles on $\init$ with noise dimension spaces $\kil_1$
and $\kil_2$, satisfying the quantum stochastic differential
equations~\eqref{qsdel}, and let $\sgp{U_t}$ be the unitary QS
cocycle on $\init$ with noise dimension space $\kil:=\kil_1\op\kil_2$
satisfying the QS differential equation~\eqref{qsde} where
$[F^\mu_\nu]$ is given by~\eqref{F1 box F2}. Then,
\begin{equation}
\label{starstar}
U^{(1,2)}_n (t) \to U_t \text{ as } n \to \infty,
\end{equation}
in the strong operator topology on $B(\init\ot\FFock)$, for each
$t\geq 0$.
\end{thm}
\begin{proof}
Let $t\in\Rplus$. First note that, since $U_t$ is unitary and each
$U^{(1,2)}_n (t)$ is unitary and so a contraction, it suffices to
prove that $U^{(1,2)}_n (t) \to U_t$ in the weak operator topology.
Also, the uniform boundedness of the operators $ U^{(1,2)}_n (t)$
means that it suffices to fix $u,v\in\init$ and $f=
\binom{f^1}{f^2}, g= \binom{g^1}{g^2}\in\Step'_\kil \subset
L^2(\Rplus; \kil = \kil_1\op\kil_2)$, and prove the following:
\begin{multline}
\label{suffices}
\bip{u\ot\vp(f_{[0,t[})}{U^{(1,2)}_n (t)\,
v\ot\vp(g_{[0,t[})}  \to
\\
 \bip{u\ot\vp(f_{[0,t[})}{U_t\, v\ot\vp(g_{[0,t[})}.
\end{multline}
By the semigroup representation~\eqref{semigp rep},
\begin{equation}
\label{RHS} \text {R.H.S. of~\eqref{suffices} } = \ip{u}
{P_{t_1-t_0}^{f(t_0),g(t_0)} \cdots \
P_{t_{m+1}-t_m}^{f(t_m),g(t_m)}v},
\end{equation}
where $t_0 = 0$, $t_{m+1} = t$ and $\{ t_1 < \cdots < t_m \} \subset
\mathbb{D}$ are the points in $]0,t[$ (if any) where $f$ or $g$ has
a discontinuity. For $n$ greater than the orders of the step
functions $f$ and $g$, Lemma~\ref{lemma: W} implies that the L.H.S.
of~\eqref{suffices} equals
\begin{multline*}
 \Bip{u}{( \oneP_{2^{-n}}^{f^1(t_0),g^1(t_0)}\,
\twoP_{2^{-n}}^{f^2(t_0),g^2(t_0)} )^{[2^n (t_1 - t_0)]}
\cdots \\ \cdots
(
\oneP_{2^{-n}}^{f^1(t_m),g^1(t_m)}\,
\twoP_{2^{-n}}^{f^2(t_m),g^2(t_m)}
)^{[2^n (t_{m+1} - t_m)]} \\
(
\oneP_{(t-t^n_0)}^{f^1(t_m),g^1(t_m)}\,
\twoP_{(t-t^n_0)}^{f^2(t_m),g^2(t_m)}
)
\, v}.
\end{multline*}
The Lie product formula~\eqref{original} and the joint continuity
of operator composition on bounded sets, therefore
implies that
\begin{equation}
\label{LHS limit} \lim_{n\to\infty} (\,\text{L.H.S.
of~\eqref{suffices}}\,) = \ip{u} {Q_{t_1-t_0}^{f(t_0),g(t_0)} \cdots
\ Q_{t_{m+1}-t_m}^{f(t_m),g(t_m)}v},
\end{equation}
where $\sgp{Q^{c,d}_t}$ is the semigroup generated by
$G^{(1)}_{c^1,d^1} + G^{(2)}_{c^2,d^2}$.
Now
\begin{multline*}
G^{(1)}_{c^1,d^1} + G^{(2)}_{c^2,d^2} = \\
\oneK + \oneL^{c^1} + \oneM_{d^1} + \oneW^{c^1}_{d^1}
-\tfrac{1}{2}(\norm{c^1}^2 + \norm{d^1}^2)I_\init \\
+ \twoK + \twoL^{c^2} + \twoM_{d^2} + \twoW^{c^2}_{d^2}
-\tfrac{1}{2}(\norm{c^2}^2 + \norm{d^2}^2)I_\init \\
= K + L^{c} + M_{d} + W^{c}_{d} -\tfrac{1}{2}(\norm{c}^2 +
\norm{d}^2)I_\init,
\end{multline*}
which, by~\eqref{Gcd}, is the generator of the semigroup $\sgp{P^{c,d}_t}$,
 for each $c,d\in\kil$.
The result therefore follows
from~\eqref {LHS limit} and~\eqref{RHS}.
\end{proof}

\begin{rem}
The joint continuity of operator composition on bounded sets
also gives a straightforward extension of this result to
time-homogeneous adapted unitary evolutions $\evol{U_{s,t}}$:
\[
U^{(1,2)}_n(s,t) \to U_{s,t} \ \text{ as } n \to \infty,
\]
in the strong operator topology, for all $0\leq s \leq t$,
where
\[
U^{(1,2)}_n(s,t):=
 \Uonetwo_{s,s^n_1}
 \big( \Uonetwo_{s^n_1,s^n_2} \Uonetwo_{s^n_2,s^n_3} \cdots
\Uonetwo_{t^n_{-1},t^n_0} \big) \Uonetwo_{t^n_0,t}.
\]
\end{rem}

\section{Extensions and an example}
\label{Section: Extensions and example}

The quantum stochastic product formula also holds for Markov-regular QS
\emph{contraction} cocycles, with the same proof, since these are
equally characterised as contraction processes which satisfy a QS
differential equation of the form~\eqref{qsde},
in other words Theorem~\ref{1.1*} still holds; contractivity of the
cocycle corresponds precisely to the matrix of coefficients of the
QS differential equation satisfying the inequality
\[
\sum_{\mu,\nu\geq 0} \bip{v^\mu}{\big( (F^\nu_\mu)^* + F^\mu_\nu +
\sum\nolimits_{j\geq 1} (F^j_\mu)^* F^j_\nu\, \big) v^\nu} \ \leq 0,
\]
equivalently,
\[
\sum_{\mu,\nu\geq 0} \bip{v^\mu}{\big( (F^\nu_\mu)^* + F^\mu_\nu +
\sum\nolimits_{j\geq 1} F_j^\mu (F_j^\nu)^*\, \big) v^\nu} \ \leq 0,
\]
 for all
 $v = (v^\mu)_{\mu\geq 0} \in
 \init \ot (\Comp \op \kil) = \bigoplus_{\mu\geq 0} \init$
 (\cite{Fag}, \cite{MoP}),
 cf.\ the equalities~\eqref{structure} for the unitary case.
 However in this case the convergence of the Trotter
products is only assured in the weak operator topology (or rather in
the hybrid norm $\FFock_\kil$-weak operator topology,
see~\cite{LWmathproc}).

Using an extension of the standard Trotter product formula
to products of several semigroups, our QS product formula
extends to cover a finite number of QS unitary (or
contraction) cocycles $\sgp{U_t^1}$, ... , $\sgp{U_t^p}$. The
coefficient matrix for the QS differential equation of the resulting
QS cocycle will then have the block matrix form:
\[
\begin{bmatrix}
\oneK + \cdots + \peeK & \oneM & \twoM & \cdots & \peeM \\
\oneL & \oneW - \oneI & 0 & \cdots & 0 \\
\twoL & 0 & \twoW -\twoI & \ddots & \vdots \\
\vdots & \vdots & \ddots & \ddots & 0 \\
\peeL & 0 & \cdots & 0 & \peeW -\peeI
\end{bmatrix}.
\]

\begin{eg}
The cocycles considered in~\cite{PSrandomTrotter} are the random
unitaries defined by
\[
U^l(s,t,\omega^l) = e^{i(\omega^l(t)-\omega^l(s))H_l},  \quad 0\leq
s \leq t,
\]
for $l=1,2$, where $\omega^1$ and $\omega^2$ are paths of two
independent classical Brownian motions $\sgp{B^1_t}$ and
$\sgp{B^2_t}$, and $H_1$ and $H_2$ are selfadjoint operators on a
Hilbert space $\init$. Recall the notation~\eqref{tnk}. By viewing
$\omega:=(\omega^1,\omega^2)$ as a path of the two-dimensional
Brownian motion $\sgp{(B^1_t,B^2_t)}$ with probability space
$\Omega$, and
\[
\Uonetwo(s,t,\omega) := e^{i(\omega^1(t)-\omega^1(s))H_1}\,
e^{i(\omega^2(t)-\omega^2(s))H_2}, \quad 0\leq s \leq t,
\]
as multiplication operators on $L^2(\Omega; \init)$, it is
shown---under the assumption that the nonnegative symmetric operator
$(H_1)^2 + (H_2)^2$ is selfadjoint---that the sequence
$(U_n^{(1,2)}(s,t,\omega))_{n\geq 1}$ of unitary operators:
\[
\Uonetwo(s,s_1^n,\omega)\big(\Uonetwo(s_1^n,s_2^n,\omega)\cdots
\Uonetwo(t_{-1}^n,t_0^n,\omega)\big)\Uonetwo(t_0^n,t,\omega)
\]
weak-operator converges to the unique contraction-operator valued process
satisfying the classical stochastic differential equations
\begin{multline*}
d_t\, U(s,t,\omega)v =
 i\, U(s,t,\omega)H_1v\, dB^1_t(\omega) + i\,
U(s,t,\omega)H_2v\,
dB^2_t(\omega) \\
 -\tfrac{1}{2}\, U(s,t,\omega)\big((H_1)^2+(H_2)^2\big)v\, dt
\end{multline*}
($v\in\Dom ((H_1)^2 + (H_2)^2)$), and that if the process $\evol{U(s,t,\omega)}$ is
unitary-valued then the convergence is strong.

\begin{rem}
Under the assumption of selfadjointness of $\sum_{l=1}^d (H_l)^2$,
the corresponding result is shown to hold for any finite number of
such unitary cocycles
$\sgp{U^l(s,t,\omega)}$, $l = 1,\ldots,d$.
\end{rem}

This may be recast in our quantum stochastic setting by identifying
the Brownian motion $\sgp{B^l_t}$ with the quantum stochastic process
$\sgp{Q^l_t:=(A^{l\, *}_t + A^l_t)^-}$ on $\FFock^{(l)}$ (where the
bar denotes operator closure), and setting
\[
U^l_t = e^{iH_l(t)}, \quad t\geq 0, \] where $H_l(t)$ is the
selfadjoint operator $H_l\ot Q^{(l)}_t$ on $\init \ot \FFock^{(l)}$,
for $l=1, \cdots , d$. Here however the coefficients of the
corresponding differential equation are unbounded,
with coefficients having block matrix form
\[
\elF =
\begin{bmatrix}
-\frac{1}{2}(H_l)^2 &  iH_l \\
iH_l & 0 \\
\end{bmatrix}, \quad l=1, \cdots , d,
\]
and
\[
F =
\begin{bmatrix}
-\frac{1}{2} K &  iH_1 & \cdots & iH_d \\
iH_1 & 0 & \cdots & 0 \\
\vdots & \vdots & \ddots & \vdots \\
iH_d & 0 & \cdots  & 0
\end{bmatrix}
\text{ where }
K = (H_1)^2 + \cdots (H_d)^2.
\]
This class of example is discussed in more detail
in~\cite{more Trotter}.
\end{eg}

\section{Concluding remarks}
\label{Section: concluding}

The methods of this paper extend to more general QS cocycles.
Firstly, quantum stochastic Trotter product formulae may be obtained
for completely contractive QS cocycles on operators spaces and
completely positive QS cocycles on $C^*$-algebras. Secondly,
\emph{strongly continuous} (as opposed to Markov-regular) QS
cocycles may be shown to satisfy the QS Trotter product formula
developed here. Conversely, the formula may be used to construct QS
cocycles from simpler cocycles with lower dimensional noises. This
yields potential applications to multidimensional diffusions. The
basic conditions under which Trotter products converge is that the
sum of sufficiently many pairs of associated semigroup generators
are pregenerators of contraction semigroups. Here the assumption of
analyticity of the expectation semigroups of the constituent cocycles
helps (\cite{LS holomorphic}). As in the Markov-regular case, strong
(as opposed to weak) operator convergence holds for Trotter products
of isometric QS cocycles if and only if the limiting cocycle is
isometric. Coisometry, on the other hand, is equivalent to isometry
of the \emph{dual} cocycle (see~\cite{L}). Unitarity for strongly continuous
contraction cocycles is assured when the cocycle satisfies a QS
differential differential whose coefficients satisfy \emph{Feller
conditions} (see~\cite{SinhaGoswami}).

All these extensions are treated in~\cite{more Trotter}. They are
facilitated by characterisations of QS cocycles in terms of (a small
number of) their associated semigroups (\cite{Accardi},
\cite{LWCBonOS}). Here Skeide's multidimensional generalisation
(\cite{Skeide}) of a theorem of Parthasarathy and Sunder
(\cite{PSu}) plays a key role.
 The homomorphic property of Trotter
product limits of Evans-Hudson type cocycles on operator algebras is
tackled in~\cite{DGS}.

\section*{Acknowledgements}
Both authors acknowledge support from the UK-India Education and
Research Initiative (UKIERI). KBS is grateful for the support of a
Bhatnagar Fellowship from CSIR, India, and JML is grateful for the
hospitality of the Delhi Centre of the Indian Statistical Institute,
where this collaboration began,
and the J.N.~Centre for Advanced Research, Bangalore.

\end{document}